\documentclass[11pt]{article}

\usepackage[top=3cm,bottom=3cm,left=3.2cm,right=3.2cm,headsep=10pt,a4paper]{geometry} 

\usepackage{amsmath}
\usepackage{amsthm}
\usepackage{amssymb}
\usepackage[applemac]{inputenc}
\usepackage{avant}
\usepackage{mathpazo}

\usepackage{microtype}

\usepackage{todonotes}
\usepackage{tikz}
\usepackage{caption}
\usepackage{multicol}
\usepackage[absolute,overlay]{textpos}

\usepackage{enumerate}
\makeatletter

\newtheoremstyle{mystyle}{}{}{\slshape}{2pt}{\scshape}{.}{ }{} 

\newtheorem{thm}{Theorem}[section]

\newtheorem{prop}[thm]{Proposition}
\newtheorem{lemme}[thm]{Lemma}
\newtheorem{lemma}[thm]{Lemma}

\newtheorem{fact}[thm]{Fact}

\theoremstyle{definition}
\newtheorem{defi}[thm]{Definition}
\newtheorem{definition}[thm]{Definition}
\theoremstyle{mystyle}

\theoremstyle{remark}
\newtheorem{rem}[thm]{Remark}

\newcommand{\monster}{\mathcal U}

\newcommand{\ordI}{\mathcal I}

\newcommand{\rel}[1]{\mathrel{#1}}

\newcommand{\ignore}[1]{}

\DeclareMathOperator{\tp}{tp}
\DeclareMathOperator{\dlr}{opD}
\DeclareMathOperator{\dpr}{dp-rk}
\DeclareMathOperator{\str}{st-dim}

\DeclareMathOperator{\op}{op}
\DeclareMathOperator{\aut}{Aut}

\def\indsym#1#2{%
 \setbox0=\hbox{$\m@th#1x$}%
 \kern\wd0%
 \hbox to 0pt{\hss$\m@th#1\mid$\hbox to 0pt{$\m@th#1^{#2}$\hss}\hss}%
 \lower.9\ht0\hbox to 0pt{\hss$\m@th#1\smile$\hss}%
 \kern\wd0}

\def\nindsym#1#2{%
 \setbox0=\hbox{$\m@th#1x$}%
 \kern\wd0%
 \hbox to 0pt{\hss$\m@th#1\not$\kern1.4\wd0\hss}
 \hbox to 0pt{\hss$\m@th#1\mid$\hbox to 0pt{$\m@th#1^{#2}$\hss}\hss}%
 \lower.9\ht0\hbox to 0pt{\hss$\m@th#1\smile$\hss}%
 \kern\wd0}

\title{Linear orders in NIP structures}
\author{Pierre Simon\footnote{Partially supported by NSF (grant no. 1665491) and a Sloan fellowship.}}

\date{}

\begin{document}
\maketitle

\begin{abstract}
	We show that every unstable NIP theory admits a $\bigvee$-definable linear quasi-order, over a finite set of parameters. In particular, if the theory is $\omega$-categorical, then it interprets an infinite linear order. This partially answers a longstanding open question.
\end{abstract}

\section{Introduction}

A first order structure is \emph{NIP} if every family of uniformly definable sets has finite VC-dimension. We like to think of NIP structures as being \emph{geometric}, indeed the classical examples comprise algebraically closed fields (the domain of algebraic geometry), real closed fields (semi-algebraic geometry), the field $\mathbb Q_p$ of $p$-adic numbers and algebraically closed valued fields (non-archimedean geometry). This class contains that of stable structures, for which we now have an extremely rich theory (see \cite{Sh:c}, \cite{PillayBook}).

In his paper \cite{Sh10}, Shelah introduced NIP theories and proved that any unstable NIP theory is SOP, that is admits a definable partial order with infinite chains. A longstanding open question asks whether this can be strengthened to an interpretable infinite linear order\footnote{We are not aware of any occurrence of this question in print. It seems to have been raised independently by several people, including at least Shelah and Hrushovski, from whom we first heard it.}. In this paper, we give a positive answer to a weaker form of this question: we find a $\bigvee$-definable equivalence relation such that the quotient by it is infinite and linearly ordered: see Theorem \ref{th:main}. In the case of $\omega$-categorical theories, we obtain a \emph{bona fide} interpretable linear order.

In fact, we show slightly more. Following \cite{GuinHill}, we define the op-dimension of a type as a variation on the dp-rank which only sees order-like dimensions. This dimension precisely gives the number of independent linear orders that one can define on a type. In the last section, we then define stable dimension as a counterpart to op-dimension, but this is not used elsewhere in the paper.

\smallskip
There is an important difference between our result and Shelah's theorem, giving a partial order. The existence of a partial order is a  non-structure result. Since partial orders can be arbitrarily complicated, it gives no positive information on the models of the theory. A linear order however is a much more constrained object. In fact, we hope that this theorem could open up a new perspective on NIP theories. It shows that NIP is a more structured world than was thought before and makes it reasonable to expect classification statements and analyses similar to those for stable (or superstable) theories, where linear orders would be explicitly present. Indeed we would like the linear orders to have similar role in NIP theories (or subclasses of it) as for instance strongly minimal sets play in the study of $\omega$-stable structures. Isomorphism types of linear orders could replace dimensions of regular types (or rather complement them, since an NIP theory can have  stable components). Of course, this still seems far away. A natural special case to initiate this program is the case of $\omega$-categorical structures. This will be studied in future works, starting with \cite{rank_one} which deals with $\omega$-categorical structures of thorn rank 1 and completely classifies the primitive ones.

\section{Preliminaries}

Throughout this paper, $T$ is a complete first order theory in a language $L$. We let $\monster$ be a monster model, which is $\bar \kappa$-saturated and $\bar \kappa$-strongly homogeneous for some large enough $\bar \kappa$. All sets of parameters considered have size smaller that $\bar \kappa$.

We use the notation $\phi^0$ to mean $\neg \phi$ and $\phi^1$ to mean $\phi$.

Letters such as $a,b,c$ usually denote finite tuples of variables, whereas $A,B,C$ denote small subsets of $\monster$.

The concatenation of two sequences $I$ and $J$ will be denoted $I+J$ or $I{}^\frown J$. The first notation is used for indiscernible sequences and the second one for concatenation of tuples or families of any kind.

\medskip\noindent
\underline{\textbf{Assumption}}: Throughout this paper, we assume that $T$ is NIP.


\subsection{Invariant types}

By an $A$-invariant type, we mean a global type $p$ which is invariant under automorphisms fixing $A$ pointwise. If $p(x)$ and $q(y)$ are both $A$-invariant, we can define the type $p(x)\otimes q(y)$ whose restriction to any set $B\supseteq A$ is $\tp(a,b/B)$, where $b\models q|B$ and $a\models p|B b$. It is also an $A$-invariant type. A Morley sequence of $p$ over $A$ is a sequence $I=(a_i:i\in \mathcal I)$ such that for each $i\in \mathcal I$, $a_i\models p|Aa_{<i}$. A Morley sequence of $p$ over $A$ is indiscernible over $A$ and all Morley sequences of $p$ over $A$ indexed by the same order have the same type over $A$.

Two invariant types $p(x)$ and $q(y)$ \emph{commute} if $p(x)\otimes q(y)= q(y)\otimes p(x)$.


\subsection{Indiscernible sequences}
We set here some terminology concerning indiscernible sequences that we copy from \cite{decomposition}.

Sequences $(I_i:i<\alpha)$ are \emph{mutually indiscernible} if each $I_i$ is indiscernible over $I_{\neq i}$.

The EM-type of an indiscernible sequence $I$ is the set $\{p_n:n<\omega\}$, where $p_n=\tp(a_1,\ldots,a_n)$ for some/any elements $a_1<_I \cdots <_I a_n$ of $I$.

If $I$ is an indiscernible sequence, we let $\op(I)$ denote the sequence $I$ indexed in the opposite order. If $I$ is an endless indiscernible sequence and $T$ is NIP, let $\lim(I)$ denote the limit type of $I$: the global $I$-invariant type defined by $\phi(x)\in \lim(I)$ if $\phi(I)$ is cofinal in $I$. Observe that if $\op(I_1)$ is a Morley sequence of $\lim(I)$ over $I$, then $I+I_1$ is indiscernible.

A cut $\mathfrak c =(I_0,I_1)$ of $I$ is a pair of subsequences of $I$ such that $I_0$ is an initial segment of $I$ and $I_1$ the complementary final segment, {\it i.e.}, $I=I_0+I_1$. If $J$ is a sequence such that $I_0+J+I_1$ is indiscernible, we say that $J$ \emph{fills} the cut $\mathfrak c$. To such a cut, we can associate two limit types: $\lim(I_0)$ and $\lim(\op(I_1))$ (which are defined respectively if $I_0$ and $\op(I_1)$ have no last element). The cut $(I_0,I_1)$ is \emph{Dedekind} if both $I_0$ and $\op(I_1)$ have infinite confinalities, in particular are not empty.


We now recall the important theorem about shrinking of indiscernibles and introduce a notation related to it (see e.g. \cite[Chapter 3]{NIPbook}).

\begin{defi}
A finite convex equivalence relation on $\mathcal I$ is an equivalence relation $\sim$ on $\mathcal I$ which has finitely many classes, all of which are convex subsets of $\mathcal I$.
\end{defi}

\begin{prop}[Shrinking of indiscernibles]\label{shrinking1}
Let $( a_t)_{t\in \mathcal I}$ be an indiscernible sequence. Let $d$ be any tuple and $\phi(y_0,..,y_{n-1};d)$ a formula. There is a finite convex equivalence relation $\sim_\phi$ on $\mathcal I$ such that given: 

-- $t_0<\ldots<t_{n-1}$ in $\mathcal I$;

-- $s_0<\ldots<s_{n-1}$ in $\mathcal I$ with $t_k \sim_\phi s_k$ for all $k$;

\noindent
we have $\phi(a_{t_0},..,a_{t_{n-1}};d) \leftrightarrow \phi(a_{s_0},...,a_{s_{n-1}};d)$.

Furthermore, there is a coarsest such equivalence relation.
\end{prop}

Given $A$, $I=( a_t)_{t\in \mathcal I}$, $\phi(y_0,\ldots,y_{n-1};d)$ as above, we let $\textsf T(I,\phi)$ denote the number of equivalence classes in the coarsest $\sim_\phi$ given by the proposition. By compactness, the number $\textsf T(I,\phi)$ is bounded by an integer depending only on $\phi(y_0,\ldots,y_{n-1};z)$.

If $I\subseteq J$ are indiscernible sequences and $A$ is any set of parameters, we write $I \unlhd_A J$ if for every $\phi(y_0,\ldots,y_{n-1};d)\in L(A)$, we have $\textsf T(I,\phi)=\textsf T(J,\phi)$. Intuitively, formulas with parameters in $A$ do not alternate more on $J$ than they do on $I$.

Note the following special cases:

\begin{itemize}
\item If $I$ is indiscernible over $A$, then $I\unlhd_A J$ simply means that $J$ is $A$-indiscernible and contains $I$.

\item If $I$ is without endpoints, $I \unlhd_A I_0+I+I_1$ is equivalent to the statement that $I_0$ is a Morley sequence in $\lim(\op(I))$ over $IA$ and $\op(I_1)$ is a Morley sequence in $\lim(I)$ over $AII_0$.

\end{itemize}

The following will be used repeatedly without mention.

\begin{lemme} If $I=(a_i:i \in \mathcal I)$ is indiscernible, where the indexing order $\mathcal I$ is dense without endpoints, then given any $\mathcal I \subseteq \mathcal J$ and any set $A$ of parameters, we can find $J=(a_i:i\in \mathcal J)$ extending $I$ such that $I\unlhd_A J$.
\end{lemme}
\begin{proof}
We give two arguments for this. For the first one, let $M$ be a model containing $I$ and $A$.  Assume that $I$ is a sequence of $k$-tuples. Expand $M$ by adding a $k$-ary predicate $P(x)$ to name the sequence $I$ and a $2k$-ary predicate $\leq_P$ for the order on that sequence. Let $M_P$ be the resulting structure. Let $N_P$ be a $|J|^+$-saturated elementary extension of $M_P$. Then $P(x)$ names a sequence $K$, ordered by $\leq_P$ which extends $I$. We then have, in the original language, $I \unlhd_A K$ since the fact that a formula has a certain number of alternations is first order expressible in $L_P$. By saturation, we can find a subsequence $J$ of $K$ with the right order type.

We can also build $J$ explicitly as follows: for every cut $\mathfrak c=(I_0,I_1)$ of $I$, by density of $I$ either $I_0$ has no last element or $I_1$ has no first element. Assume for example the latter. Let $J_{\mathfrak c}$ be a Morley sequence of $\lim(\op(I_1))$ over everything constructed so far which is indexed by the subsequence of $\mathcal J$ which lies in the cut of $\mathcal I$ corresponding to $\mathfrak c$. Doing this iteratively for all cuts of $I$ and adding all those sequences to $I$, we obtain $J$ as required.
\end{proof}

\subsection{Dp-rank}

The dp-rank will not be used in this paper, but we will define variations of it and hence it seems useful to recall its definition and some of its properties.

\begin{defi}
\index{dp-rank}
Let $\pi$ be a partial type over a set $A$, and let $\kappa$ be a (finite or infinite) cardinal. We say $\dpr(\pi,A)< \kappa$ if for every family $(I_t : t<\kappa)$ of mutually indiscernible sequences over $A$ and $b\models \pi$, there is $t<\kappa$ such that $I_t$ is indiscernible over $Ab$.

If $b\in \monster$, then $\dpr(b/A)$ stands for $\dpr(\tp(b/A),A)$.
\end{defi}

A theory $T$ is NIP if and only if we have $\dpr(\pi,A)<|T|^{+}$ for every finitary type $\pi$ (\cite[Observation 4.13]{NIPbook}).

The term \emph{rank} used for dp-rank is misleading as the dp-rank is a cardinal and not an ordinal. (Strictly speaking, it is not a cardinal, since we only defined $\dpr(\pi,A)<\kappa$ and not $\dpr(\pi,A)=\kappa$. This is due to a problem at limit cardinals: we can have say $\dpr(\pi,A)<\aleph_0$ and yet $\dpr(\pi,A)\geq n$ for each $n<\omega$. Some authors write this as $\dpr(\pi,A)=\aleph_0^{-}$.) The reason for it is historical: Shelah gave a more general definition of dp-ranks in \cite{Sh863}, which were indeed ordinals. The definition we use now was extracted from that paper in \cite{Us} and the name stayed. In the following sections, we will define two variations on the dp-rank, one from \cite{GuinHill} which sees only the order component, and a new one which sees only the stable component. We will call them \emph{dimensions} instead of \emph{ranks}.

\smallskip
Some properties of dp-rank in NIP theories (of unequal difficulties):

\begin{itemize}
	\item If $A\subseteq B$ and $\pi$ is over $A$, then $\dpr(\pi,A)=\dpr(\pi,B)$.
	\item Given $\pi$ a partial type over $A$ and let $\kappa$ be any cardinal. Then we have $\dpr(\pi,A)< \kappa$ if and only if for any family $(I_t :t \in X)$ of sequences, mutually indiscernible over $A$ and any $b\models \pi$, there is $X_0 \subseteq X$ of size $<\kappa$ such that $(I_t : t\in X\setminus X_0)$ are mutually indiscernible over $Ab$.
	\item (Additivity) Let $a,b \in \monster$, $A$ a small set of parameters and $\kappa_1,\kappa_2$ be two cardinals such that $\dpr(b/A) < \kappa_1$ and $\dpr(a/Ab)< \kappa_2$, then \[\dpr(a,b/A)< \kappa_1+\kappa_2-1.\]
\end{itemize}

The first bullet is relatively straightforward. The second one is from \cite{KOU}, as well as the third, which follows from it. Proofs can be also found in \cite[Section 4]{NIPbook}.

\section{Indiscernible sequences stable over a set}

Recall that we assume $T$ to be NIP.

In \cite{distal} we introduced the idea that there are two minimal ways in which an indiscernible sequence $I$ can fail to be indiscernible over a tuple $a$: either some formula $\phi(x;a)$ changes truth value at one cut of the sequence $I$, or there is an element $b\in I$ such that some formula takes a different truth value on $b$, but removing $b$ from $I$ yields an indiscernible sequence over $a$. Distal theories are exactly those for which the second behavior never happens. The following fact from \cite{distal} says that this second behavior cannot happen on a large subset of $I$.

\begin{fact}[\cite{distal}, Theorem 3.30, Corollary 3.32]\label{fact:finite cofinite}
	Let $I_1+I_2+I_3$ be indiscernible, with $I_1,I_3$ infinite without endpoints. Write $I_2=(b_i:i\in \mathcal I)$. Assume that $I_1+I_3$ is indiscernible over $A$, then:
	\begin{enumerate}
		\item if $\phi(x;a)\in L(A)$, then $\{i\in \mathcal I : \models \phi(b_i;a)\}$ is finite or co-finite in $\mathcal I$;
		\item there is $\mathcal J\subseteq \mathcal I$ of size $\leq |A|+|T|$ such that $I_1+(a_i:i\in \mathcal I \setminus \mathcal J)+I_3$ is indiscernible over $A$.
	\end{enumerate}
\end{fact}

\begin{prop}\label{prop:stable seq}
Let $I$ be an endless densely ordered indiscernible sequence and $A$ a set of parameters. The following are equivalent:

\begin{enumerate}
\item there are infinite endless sequences $I_0,I_1$ such that $I_0+I+I_1$ is indiscernible and $I_0+I_1$ is indiscernible over $A$;
\item if $I_0$ is a Morley sequence in $\lim(\op(I)/IA)$ and $\op({I_1})$ a Morley sequence in $\lim(I/II_0A)$ (equivalently, $I\unlhd_A I_0+I+I_1$), then $I_0+I_1$ is indiscernible over $A$;
\item if $I\unlhd_A J$, then $J\setminus I$ is indiscernible over $A$;
\item for every formula $\phi(x_1,\ldots,x_n)\in L(A)$, there is a finite set $I_\phi \subseteq I$ and a truth value $\mathbf t$ such that for every $a_1<\cdots <a_n$ in $I\setminus I_\phi$, $\models \phi^{\mathbf t}(a_1,\ldots,a_n)$.
\end{enumerate}

\end{prop}
\begin{proof}
The implications (4)$\to$(3)$\to$(2)$\to$(1) are straightforward: (4) implies (3) as limit types are unaffected by removing finitely many points from a sequence, (2) is a special case of (3) and (1) follows directly from (2). The implication (1)$\to$(4) is proved as in \cite[Corollary 3.32]{distal}.
\end{proof}

\begin{defi}
If the conditions of the proposition are satisfied, we say that $I$ is stable over $A$.

If $I$ is any infinite indiscernible sequence (not necessarily densely ordered), we say that it is stable over $A$ if condition (1) above is satisfied. This implies that (4) also holds. If $I$ has no endpoints, then (2) also holds.
\end{defi}

Note that if $I\subseteq J$ are two endless indiscernible sequences, $I$ coinitial and cofinal in $J$, then $I$ is stable over some set $A$ if and only if $J$ is stable over $A$. This follows from (2) above since the  limit types of $I$ and $J$ (from both sides) are the same.

Also, it follows from Fact \ref{fact:finite cofinite} that if $I$ is stable over $A$, then there is $I'\subseteq I$, $|I'|\leq |T|+|A|$ such that $I\setminus I'$ is indiscernible over $A$.

Finally, we note that this definition first appeared in \cite{GuinHill} under the name \emph{almost indiscernible sequence}. We use a different terminology, mainly because the next definition of mutually stable sequences does not coincide with almost mutually indiscernible sequences from \cite{GuinHill}.

\begin{prop}\label{prop:mutually stable}
Let $(I_i:i<\alpha)$ be a family of endless densely ordered indiscernible sequences and $A$ a set of parameters. The following are equivalent:

\begin{enumerate}
\item there are infinite endless sequences $J_i^0,J_i^1$, $i<\alpha$, such that the sequences $J_i^0+I_i+J_i^1$ are indiscernible and $J_i^0+J_i^1$ are mutually indiscernible over $A$;
\item if for each $i<\alpha$, $J_i^0$ realizes $\lim(\op(I_i)/I_{<\alpha}AJ_{<i}^0 J_{<i}^1)$ and $\op({J_i^1})$ realizes $\lim(I/I_{<\alpha}AJ_{\leq i}^0 J_{<i}^1)$, then the sequences $(J_i^0+J_i^1:i<\alpha)$ are mutually indiscernible over $A$;
\item if we construct inductively $I_i\unlhd_{AI_{<\alpha}J_{<i}} J_i$, then the sequences $(J_i\setminus I_i:i<\alpha)$ are mutually indiscernible over $A$;
\item for every formula $\phi(x^1_1,..,x^1_{n_1};x^2_1,..,x^2_{n_2};\ldots;x^{k}_1,..,x^k_{n_k})\in L(A)$ and indices $i_1>\cdots>i_k$ there is a truth value $\mathbf t$, a finite set $I^{1} \subseteq I_{i_1}$ such that for any $a^1_1<\cdots<a^1_{n_1} \in I_{i_1} \setminus I^1$, there is a finite set $I^2 \subseteq I_{i_2}$ such that for any $a^2_1<\cdots<a^2_{n_2}\in I_{i_2}\setminus I^2$, there is a finite set $I^3\subseteq I_{i_3}$ such that $\ldots\ldots$ for any $a^k_1<\cdots <a^k_{n_k}\in I_{i_k}\setminus I^k$, we have $\models \phi^{\mathbf t}(a^1_1,..,a^1_{n_1};a^2_1,..,a^2_{n_2};\ldots;a^{k}_1,..,a^k_{n_k})$.

\end{enumerate}

\end{prop}
\begin{proof}
The implications (4)$\to$(3)$\to$(2)$\to$(1) are as above. We show (1)$\to$(4). Assume (1) and take the sequences $J_ i^0,J_i^1$ to be countable. Assume that (4) fails, as witnessed by some formula $\phi$ and indices $i_1>\cdots>i_k$. For simplicity of notations, assume $k=2$, $(i_1,i_2)=(1,0)$ and $\phi=\phi(x^1,x^0)$ (we changed the variable name from $x^2$ to $x^0$ to improve readability). The general case is similar. Let $\mathbf t$ be the truth value of $\phi(a^1,a^0)$ for some/any $a^1\in J_1^0, a^0\in J_0^0$. Then we can find an infinite $I^1\subseteq I_1$ such that for all $a\in I^1$, there is an infinite $I^0_a\subseteq I_0$ such that for all $b\in I^0_a$, we have \[\neg \phi^{\mathbf t}(a,b).\]
By compactness, we can increase the sequences $I_1$ and $I_0$ and assume that $I^1$ and each $I^0_a$ have size $\geq |T|^+$. As the sequence $J_1^0+J_1^1$ is indiscernible over $J_0^0+J_0^1$ and the latter is countable, by Fact \ref{fact:finite cofinite}, there is a subset $I^1_*\subseteq I_1$ of size $\leq |T|$ such that $J_1^0+(I_1\setminus I^1_*)+J_1^1$ is indiscernible over $J_0^0+J_0^1$. Let $a\in  (I^1\setminus I^1_*)$. Then $J_1^0+(a)+J_1^1$ is indiscernible over $J_0^0+J_0^1$, and hence those two sequences are mutually indiscernible. Next, we can similarly find $I^0_*\subseteq I_0$ of size $\leq |T|$ such that $J_0^0+(I_0\setminus I^0_*)+J_0^1$ is indiscernible over $J_1^0+(a)+J_1^1$. So there is some $b\in I^0_a\setminus I_*^0$. Then the sequences $J_1^0+(a)+J_1^1$ and $J_0^0+(b)+J_0^1$ are mutually indiscernible. But this contradicts the construction of $I^0_a$.
\end{proof}

When the conditions in the last proposition are satisfied, we say that the family $(I_i:i<\alpha)$ is \emph{mutually stable} over $A$. Condition (1) shows that this notion does not depend on the ordering of the family. As previously, we extend this definition to arbitrary indexing orders using condition (1). This implies that condition (4) holds and, if the sequences $(I_i)$ are endless, condition (2) holds.

The following lemmas will be used repeatedly.

\begin{lemme}\label{lem:pound}
Let $I_1,\ldots,I_{m}$ be infinite sequences, mutually indiscernible over some $A$. Assume that the sequence $I_1+\cdots+I_{m}$ is indiscernible and stable over $A$. Then $I_1+\cdots+I_{m}$ is indiscernible over $A$.
\end{lemme}
\begin{proof}
Fix a formula $\phi(x_1,\ldots,x_n)$ over $A$. By Proposition \ref{prop:stable seq}(4), there is a finite set $I_\phi \subseteq I_1+\cdots+I_{m}$ and a truth value $\mathbf t$ such that for every $a_1<\cdots <a_n$ in $I_1+\cdots+I_{m}\setminus I_\phi$, $\models \phi^{\mathbf t}(a_1,\ldots,a_n)$. Now take an arbitrary $a_1,\ldots,a_n \in I_1+\cdots +I_m$. Since the sequences $I_1,\ldots,I_{m}$ are infinite and mutually indiscernible over $A$, we can find $a'_1,\ldots,a'_n\in I_1+\cdots +I_m$ having the same type as $a_1,\ldots,a_n$ over $A$ and disjoint from $I_\phi$. We know that $\phi^{\mathbf t}(a'_1,\ldots,a'_n)$ holds, so also $\phi^{\mathbf t}(a_1,\ldots,a_n)$ is true. Hence the sequence $I_1+\cdots+I_m$ is indiscernible over $A$.
\end{proof}

\begin{lemme}\label{lem:poundpound}
Let $(J_i:i<\alpha)$ be sequences, mutually stable over some $A$, where each $J_i$ can be written as $J_i = I_i^1+\cdots +I_i^k$ and the sequences $(I_i^j: i<\alpha, j\leq k)$ are infinite and mutually indiscernible over $A$. Then the sequences $(J_i:i<\alpha)$ are mutually indiscernible over $A$.
\end{lemme}
\begin{proof}
The proof is the same as that of the previous lemma using Proposition \ref{prop:mutually stable}(4) instead of Proposition \ref{prop:stable seq}(4).
\end{proof}

\section{Op-dimension}

Op-dimension was introduced by Guingona and Hill in \cite{GuinHill}. It measures the number of independent orderings that one can define on an infinite subset of a (partial) type. We give a self-contained exposition, slightly different from, but equivalent to, the one in \cite{GuinHill} (in the case of finite cardinals, see the remarks below).

\begin{definition}
	Let $A$ be any set of parameters, $\pi$ a partial type over $A$ and $\kappa$ a cardinal. We say that $\dlr(\pi,A)<\kappa$ if we cannot find:

	$\bullet$ $a\models \pi$;

	$\bullet$ a family $(I_i:i<\kappa)$ of sequences mutually indiscernible over $A$, where $I_i=(b^ i_j:j\in \ordI_i)$;
		
	$\bullet$ for each $i<\kappa$, a formula $\phi_i(x;y_i)\in L(A)$ (with $|x|=|a|$ and $|y_i|=|b^ i_j|$), such that $\{j\in \ordI_i:\models \phi_i(a;b^ i_j)\}$ is infinite and co-infinite in $\ordI_i$.
\end{definition}

It follows at once from the definition that $\dlr(\pi, A)\leq \dpr(\pi , A)$. In particular, if $T$ is NIP and $\pi$ is a partial type in finitely many variables, then $\dlr(\pi,A)<|T|^+$.

\begin{lemma}\label{lem:base change distal}
	If $A\subseteq B$ and $\pi(x)$ is a partial type over $A$, then $\dlr(\pi,A)=\dlr(\pi,B)$.
\end{lemma}
\begin{proof}
	If $a$, $(I_i:i<\kappa)$ is a witness to $\dlr(\pi,B)\geq \kappa$, then $a$, $(I'_i:i<\kappa)$ witnesses $\dlr(\pi,A)\geq \kappa$, where $I'_i$ is obtained from $I_i$ by concatenating a fixed enumeration of $B$ to the end of every element of it.

Conversely, if $(I_i:i<\kappa)$ are $A$-mutually indiscernible and witness $\dlr(\pi,A)\geq \kappa$, then there is $B'\equiv_A B$ such that those sequences are mutually indiscernible over $B'$. We then have $\dlr(\pi,B)=\dlr(\pi,B')\geq \kappa$. Hence $\dlr(\pi,B)\geq \dlr(\pi,A)$.
\end{proof}

Hence it makes sense to write $\dlr(\pi)$ to stand for $\dlr(\pi,A)$ for any $A$ over which $\pi$ is defined.

This definition of op-dimension does not appear in this form in \cite{GuinHill}, but it is easily seen by Ramsey compactness to be equivalent to the one using IRD-patterns given as Lemma 1.23 in that paper. (The only difference is that we allow it to take infinite cardinal values, whereas the definition as stated in \cite{GuinHill} only allows integers or $\infty$.)

\begin{prop}\label{prop:dlrank}
	Let $\pi(x)$ be a partial type and $\kappa$ a cardinal. The following are equivalent:
	\begin{enumerate}
		\item $\dlr(\pi)<\kappa$;
		\item for every $A$ over which $\pi$ is defined, for every family $(I_i:i<\kappa)$ of endless sequences mutually stable over $A$ and $a\models \pi$, there is $i<\kappa$ such that $I_i$ is stable over $Aa$;
		\item for every $A$ over which $\pi$ is defined, for every family $(I_i:i<\lambda)$ of endless sequences mutually stable over $A$ and $a\models \pi$, there is $X\subseteq \lambda$, $|X|<\kappa$ such that $(I_i:i\in \lambda\setminus X)$ are mutually stable over $Aa$.
	\end{enumerate}
\end{prop}
\begin{proof}
	$(1)\Rightarrow (2)$: Assume we are given a family $(I_i:i<\kappa)$ of endless sequences mutually stable over $A$ such that no $I_i$ is stable over $Aa$. Using the second bullet after Proposition \ref{shrinking1}, build inductively dense endless sequences $J_i$, so that \[I_i \unlhd J_i^0+I_i+J_i^1,\] over everything constructed so far. Then the sequences $(J_i^0+J_i^1:i<\kappa)$ are mutually indiscernible over $A$, but no $J_i^0+J_i^1$ is indiscernible over $Aa$. Since the sequences $(J_i^0,J_i^1:i<\kappa)$ are mutually indiscernible over $Aa$, we can find a formula $\phi(x;y)\in L(AJ_i^0 J_i^1)$ such that $\phi(a;b)$ holds for all $b$ in some end segment of $J_i^0$, whereas $\neg \phi(a;b')$ holds for all $b'$ in some initial segment of $J_i^1$. Adding the parameters of $\phi$ to the base and trimming the sequences gives a witness of $\dlr(\pi)\geq \kappa$.

	$(2)\Rightarrow (3)$: Assume that $\dlr(\pi,A)<\kappa$. Let $(I_i:i<\lambda)$ be mutually stable over $A$ and let $a\models \pi$. Using the second bullet after Proposition \ref{shrinking1}, build inductively dense endless sequences $J^k_i$, $i<\lambda$, $k<4$, so that \[I_i \unlhd J^0_i+J^1_i+I_i+J^2_i+J^3_i\] over everything constructed so far. Let $A'=AJ_{<\lambda}^0J_{<\lambda}^3$. The sequences $(J_i^1+J_i^2:i<\lambda)$ are mutually stable (indeed mutually indiscernible) over $A'$. Let $X\subseteq \lambda$ be the set of indices $i$ for which $J_i^1+J_i^2$ is not indiscernible over $A'a$. Since the sequences $(J_i^1,J_i^2:i<\lambda)$ are mutually indiscernible over $A'a$, Lemma \ref{lem:pound} implies that for $i\in X$, $J_i ^1 + J_i ^2$ is not stable over $A'a$. Hence by (2), $|X|<\kappa$. Now since $(J_i^0+J_i^1 , J_i^2+J_i^3 : i<\lambda)$ are mutually indiscernible over $Aa$, each sequence $J_i ^0+J_i ^1+J_i ^2+J_i ^3$ for $i\notin X$ is indiscernible over $AaJ_{\neq i}$. Hence the sequences $(J_i^0+J_i^1 , J_i^2+J_i^3 : i\in \lambda \setminus X)$ are mutually indiscernible over $Aa$ and the sequences $(I_i:i\in \lambda \setminus X)$ are mutually stable over $Aa$.

	
	$(3)\Rightarrow (1)$ is immediate.	
\end{proof}

The following is \cite[Theorem 2.2]{GuinHill}. The proof we give is very similar to the one by Guingona and Hill, except that we use mutually stable sequences instead of almost mutually indiscernible, which is a stronger notion.

\begin{prop}
	Let $A$ be any set of parameters and $a,b$ two tuples. If $\dlr(a/A)<\kappa_1$ and $\dlr(b/Aa)<\kappa_2$, then $\dlr(a,b/A)< \kappa_1 + \kappa_2 -1$.
\end{prop}
\begin{proof}
	Let $(I_i:i<\lambda)$ be mutually stable over $A$. We can find $X_1\subseteq \lambda$, $|X_1|<\kappa_1$ such that the sequences $(I_i:i\in \lambda \setminus X_1)$ are mutually stable over $Aa$. We then find $X_2\subseteq \lambda$, $|X_2|<\kappa_2$ such that the sequences $(I_i:i\in \lambda\setminus (X_1\cup X_2))$ are mutually stable over $Aab$. This proves $\dlr(a,b/A)<\kappa_1+\kappa_2-1$.
\end{proof}

\section{Constructing linear orders}


\textbf{Assumption:} All indiscernible sequences in this section are assumed to be indexed by $\mathbb Q$. This will be recalled at times. The density of the indexing order is only for convenience and could be removed it most places, however having no endpoints is often essential for the arguments to go through, as we want to be able to extend sequences on both sides by realizing limit types.

\begin{definition}
A quintuple $\mathbf u = (\pi(x),I,J,\phi;A)$ is \emph{good} if:

$\bullet$ $I=(a_i:i\in \mathbb Q)$ and $J=(b_j:j\in \mathbb Q)$ are sequences of tuples, $A$ is a small set of parameters and $\pi(x)$ is a partial type over $AIJ$;

$\bullet$ the sequence $I+J$ is indiscernible over $A$;

$\bullet$ $\phi=\phi(x;y)\in L(A)$ with $x$ the same variable as that of $\pi(x)$ and $|y|=|a_i|=|b_j|$;

$\bullet$ there is $a\models \pi(x)$ such that for each $i\in \mathbb Q$, we have $\models \phi(a;a_i)$ and for each $j\in \mathbb Q$, we have $\models \neg \phi(a;b_j)$.
\medskip
\\
We will sometimes omit $A$ from the notation if it is irrelevant.
\medskip
\\
We write $(a,I,J,\phi;A)$ for $(\tp(a/AIJ),I,J,\phi;A)$.
\medskip
\\
If $\mathbf u = (a,I,J,\phi;A)$ and $\mathbf v = (a',I',J',\phi';A')$ are good, write $\mathbf u \trianglelefteq \mathbf v$ to mean:

$\bullet$ $\phi=\phi'$;

$\bullet$ $A\subseteq A'$;

$\bullet$ $\tp(a/A)=\tp(a'/A)$;

$\bullet$ $I+J$ and $I'+J'$ have the same EM-type over $A$.
\end{definition}

\begin{definition}
	Let $\mathbf u = (a,I,J,\phi;A)$ be good and set $p(\mathbf u)=\tp(a/AIJ)$. We define the following binary relations on realizations of $p(\mathbf u)$:
	
	$\bullet$ $(a,b)\in E(\mathbf u)$ if for any sequence $K$ such that $I+K+J$ is $A$-indiscernible, if $\phi(a;d)$ (resp. $\neg \phi(a;d)$) holds for all $d\in K$, then $\phi(b;d)$ (resp. $\neg \phi(b;d)$) holds for almost all $d\in K$ (all but finitely many).
		
	$\bullet$ $(a,b)\in R(\mathbf u)$ if we cannot find a sequence $K$ such that both $(a,I+K,J,\phi;A)$ and $(b,I,K+J,\phi;A)$ are good.
\end{definition}

\begin{lemma}
	Let $\mathbf u$ be good, then $E(\mathbf u)$ and $R(\mathbf u)$ are $\bigvee$-definable relations on $p(\mathbf u)$. The relation $E(\mathbf u)$ is an equivalence relation and $R(\mathbf u)$ is reflexive, transitive and $E(\mathbf u)$-equivariant. Furthermore $(a,b)\in E(\mathbf u)$ if and only if both $(a,b)$ and $(b,a)$ are in $R(\mathbf u)$.
\end{lemma}
\begin{proof}
	Only transitivity of $R(\mathbf u)$ is not immediate from the definition. Let $a,b,c\models p(\mathbf u)$ and assume that $(a,c)\notin R(\mathbf u)$ as witnessed by $K$: that is $(a,I+K,J,\phi;A)$ and $(c,I,K+J,\phi;A)$ are good. Replacing $K$ by a Morley sequence of its limit type, we may assume that $K$ is indiscernible over $Ab$. If $\phi(b;d)$ holds for $d\in K$, then $(b,I+K,J,\phi;A)$ is good. This implies that $(b,c) \notin R(\mathbf u)$. Similarly, if $\neg\phi(b;d)$ holds for $d\in K$, then $(b,I,K+J,\phi;A)$ is good and $(a,b)\notin R(\mathbf u)$.
\end{proof}

\begin{definition}
	We say that $\mathbf u$ is linear if it is good and any two realizations of $p(\mathbf u)$ are $R(\mathbf u)$-comparable.
\end{definition}

In other words, $\mathbf u=(a,I,J,\phi;A)$ is linear if it is good and we cannot find $a,b\models p(\mathbf u)$ and $K,L$ two endless indiscernible sequences such that both $(a,I,K+J,\phi;A)$ and $(b,I+K,J,\phi;A)$ are good as well as $(a,I+L,J,\phi;A)$ and $(b,I,L+J,\phi;A)$.

If $\mathbf u$ is linear, then $R(\mathbf u)$ induces a linear order on the quotient of $p(\mathbf u)$ by $E(\mathbf u)$. It is usually easier to consider the type-definable relation $\neg R(\mathbf u)$ which induces a strict linear order on that quotient.

\begin{lemme}\label{lem_max decomposition}
Let $\mathbf u_\alpha=(a,I^0_\alpha,J^0 _\alpha,\phi_\alpha;A)$, $\alpha<\eta$, be good, where the sequences $(I^0 _\alpha+J^0 _\alpha:\alpha<\eta)$ are mutually indiscernible over $A$. There is $\mu$ with $\eta \leq \mu\leq \dlr(a/A)$ and a family $(I_\alpha,J_\alpha,\phi_\alpha)_{\alpha<\mu}$ such that:
\begin{itemize}
\item[$\bullet_0$] each $(a,I_\alpha,J_\alpha,\phi_\alpha;A)$ is good;

\item[$\bullet_1$] for $\alpha<\eta$, $\mathbf u_\alpha \unlhd (a,I_\alpha,J_\alpha,\phi_\alpha;A)$;
\item[$\bullet_2$] the sequences $(I_\alpha+J_\alpha:\alpha<\mu)$ are mutually indiscernible over $A$;
\item[$\bullet_3$] whenever $K_\alpha^0, K_\alpha^1$, $\alpha<\mu$ are sequences and $A'\supseteq A$ are such that the sequences $(I_\alpha+J_\alpha:\alpha<\mu)$ are mutually indiscernible over $A'$ and $(a,I_\alpha+K_\alpha^0,K_\alpha^1+J_\alpha,\phi_\alpha;A)$ are good, then the sequences in the family $(I_\alpha+K_\alpha^0:\alpha<\mu){}^\frown(K_\alpha^1+J_\alpha:\alpha<\mu)$ are mutually stable over $A'a$;
\item[$\bullet_4$] if $L$ is indiscernible over $AI_{<\mu}J_{<\mu}$ and the sequences $(I_\alpha+J_\alpha:\alpha<\mu)$ are mutually indiscernible over $AL$, then $L$ is stable over $AI_{<\mu}J_{<\mu}a$;

\item[$\bullet_5$] the sequences in the family $(I_\alpha:\alpha<\mu){}^\frown (J_\alpha:\alpha<\mu)$, are mutually indiscernible over $Aa$.
\end{itemize}

\end{lemme}
\begin{proof}
Call a family $(\psi_\alpha:\alpha<\mu)$ \emph{good} if each $\psi_\alpha = \psi_\alpha(x;y_\alpha)\in L(A)$, with $|x|=|a|$ and there are sequences $(I_\alpha+J_\alpha:\alpha<\mu)$ mutually indiscernible over $A$ such that each $(a,I_\alpha,J_\alpha,\psi_\alpha;A)$ is good and for $\alpha<\eta$, $I_\alpha+J_\alpha$ has same EM-type as $I^0_\alpha+J^0_\alpha$ over $A$. Any good family has length $\mu\leq \dlr(a/A)<|T|^+$. Also by compactness, an increasing union of good families is again good. Hence there is a maximal good family $(\phi_\alpha:\alpha<\mu)$ extending $(\phi_\alpha:\alpha<\eta)$. Let its goodness be witnessed by $(I_\alpha+J_\alpha:\alpha<\mu)$. Properties $\bullet_0$ and $\bullet_2$ are immediate by construction. We will prove that this family satisfies $\bullet_3$ and $\bullet_4$. We can then enforce $\bullet_5$ by building $I_\alpha \unlhd I_\alpha+I'_\alpha$ and $J_\alpha \unlhd J'_\alpha+J_\alpha$ over everything and replacing each $(I_\alpha,J_\alpha)$ by $(I'_\alpha,J'_\alpha)$. The new family still witnesses goodness and therefore also satisfies $\bullet_{1\to 4}$.

Assume that $\bullet_4$ does not hold. So there is some $L$ such that the sequences $(I_i + J_i:i<\mu)$ along with $L$ are mutually indiscernible over $A$, but $L$ is not stable over $AI_{<\mu}J_{<\mu}a$. Without loss, $L$ has no endpoints. Let $L\unlhd_{AI_{<\mu}J_{<\mu}a} L_0+L+L_1$. Then $L_0,L_1$ are mutually indiscernible over $AI_{<\mu}J_{<\mu}a$ but $L_0+L_1$ is not indiscernible over that same set. There is a formula $\phi_*(a,y;\bar d)$, $\bar d\in AI_{<\mu}J_{<\mu}L_0L_1$, which holds on an end segment of $L_0$, whereas its negation holds on an initial segment of $L_1$. Take $L'_0$ an end segment of $L_0$, $L'_1$ an initial segment of $L_1$ such that $L'_0,L'_1$ contain no element from $\bar d$. Take also end segments $I'_{\alpha}$ of $I_{\alpha}$ and initial segments $J'_{\alpha}$ of $J_{\alpha}$ for $\alpha<\mu$ such that those also do not contain any element from $\bar d$. Finally let $L''_0,L''_1$ be the sequences $L'_0,L'_1$ respectively, with $\bar d$ concatenated to every element of the sequence. Then the sequences $I'_{\alpha}+J'_{\alpha}$, $\alpha<\mu$ and $L''_0+L''_1$ are mutually indiscernible over $A$. Each $(a,I'_\alpha,J'_\alpha,\phi_\alpha;A)$ is good as is $(a,L''_0,L''_1,\phi'_*;A)$, where $\phi'_*(x;y\hat{~}\bar z)=\phi_*(x,y;\bar z)$. Also the EM-type over $A$ of each $I'_\alpha+J'_\alpha$ is the same as that of $I_\alpha+J_\alpha$. This contradicts maximality of the initial family.

From $\bullet_4$, we can deduce two seemingly stronger statements:

\begin{itemize}
\item[$\bullet_4'$] If $A\subseteq A'$, $L$ is indiscernible over $A'I_{<\mu}J_{<\mu}$ and the sequences $(I_\alpha+J_\alpha:\alpha<\mu)$ are mutually indiscernible over $A'L$, then $L$ is stable over $A'I_{<\mu}J_{<\mu}a$.

\item[$\bullet_4''$] If $A\subseteq A'$, $(L_i:i<\beta)$ is a family of sequences mutually indiscernible over $A'I_{<\mu}J_{<\mu}$ such that $(I_\alpha+J_\alpha:\alpha<\mu)$ are mutually indiscernible over $A'L_{<\beta}$, then $(L_i:i<\beta)$ are mutually stable over $A'I_{<\mu}J_{<\mu}a$.
\end{itemize}

To see that $\bullet_4'$ follows, consider the sequence $L'$ obtained from $L$ by concatenating $A'$ to each of its elements. Then $\bullet_4$ applied to $L'$ gives $\bullet_4'$. To deduce $\bullet_4''$, let $(L_i:i<\beta)$ be given as above. Construct inductively on $i<\beta$, $L_i \unlhd M_i+L_i+N_i$ over everything built so far, including all of the $L_i$'s. Then by $\bullet_4'$, $M_i+N_i$ is stable over $A'I_{<\mu}J_{<\mu}M_{\neq i}N_{\neq i}a$. It follows from Lemma \ref{lem:pound} that $M_i+N_i$ is indiscernible over that set. Thus the sequences $(M_i+N_i:i<\beta)$ are mutually indiscernible over $A'I_{<\mu}J_{<\mu}a$ as required.

We now show that $\bullet_3$ holds. So let $K^0_\alpha,K^1_\alpha$, $\alpha<\mu$, and $A'\supseteq A$ be given such that $(I_\alpha+J_\alpha:\alpha<\mu)$ are mutually indiscernible over $A'$ and $(a,I_\alpha+K_\alpha^0,K_\alpha^1+J_\alpha,\phi_\alpha;A)$ are good. Build inductively on $\alpha<\mu$, \[I_\alpha+K_\alpha^0\unlhd M^0_\alpha+(I_\alpha+K_\alpha^0)+N^0_\alpha+I'_\alpha\] and \[K^ 1_\alpha+J_\alpha\unlhd J'_\alpha+N^1_\alpha+(K^1_\alpha+J_\alpha)+M^1_\alpha,\] each $\unlhd$ being understood to hold over everything built so far (including $a$). Then the tuples $(a,I'_\alpha,J'_\alpha,\phi_\alpha; A')$ are good. Furthermore as the sequences $(I_\alpha+K_\alpha^0+K_\alpha^1+J_\alpha:\alpha<\mu)$ are mutually stable over $A'$ (as witnessed by $I_\alpha+J_\alpha$), the sequences \[(M^0_\alpha+N^0_\alpha+I'_\alpha+J'_\alpha+N^1_\alpha+M^1_\alpha:\alpha<\mu)\] are mutually indiscernible over $A'$ (property (3) of Proposition \ref{prop:mutually stable}). The family $(I'_\alpha+J'_\alpha)_\alpha$ also witnesses maximality of the family $(\phi_\alpha)_{\alpha<\mu}$, hence we can apply $\bullet_4''$ to it. We deduce that the sequences $(M^{0}_\alpha+N^{0}_\alpha:\alpha<\mu){}^\frown (N^{1}_\alpha+M^{1}_\alpha:\alpha<\mu)$ are mutually stable over $A'I'_{<\mu}J'_{<\mu}a$. By Lemma \ref{lem:poundpound}, they are mutually indiscernible over that set. Therefore $(I_\alpha+K^0_\alpha:\alpha<\mu){}^\frown(K_\alpha^1+J_\alpha:\alpha<\mu)$ are mutually stable over $A'a$.
\end{proof}

\begin{rem}
	If $\dlr(a/A)\geq \eta$, then there is a family $(I_\alpha,J_\alpha,\phi_\alpha)_{\alpha<\mu}$ satisfying $\bullet_0$ and $\bullet_{2\to 4}$ with $\mu\geq \eta$.
\end{rem}
We now come to the main technical point of the construction.

\begin{prop}\label{prop_exists linear}
	Let $\mathbf u _\alpha=(a,I_\alpha,J_\alpha,\phi_\alpha;A)$, $\alpha<\mu$ be a family satisfying $\bullet_0$ and $\bullet_{2\to 5}$ of Lemma \ref{lem_max decomposition}. Then there is $A'\supseteq A$ such that each $\mathbf v_\alpha:=(a,I_\alpha,J_\alpha,\phi_\alpha;A')$ is linear.
\end{prop}
\begin{proof}
During the proof we will often replace a sequence $K$ say by a sequence $K'$ so that $K\unlhd K+K'$ over everything that we have built so far. Thus the new sequence $K'$ is indiscernible over all parameters considered. This will never affect previous assumptions: for instance if $K$ was indiscernible over some $B$, then $K'$ has the same type as $K$ over it. If some $(b,K,L,\psi; B)$ was good, then so is $(b,K',L,\psi;B)$. If $K$ was stable over some $B$, then $K'$ is also stable over $B$. In particular, note that by $\bullet_5$, doing this to the original sequences $I_\alpha$, $J_\alpha$ does not change their type over $Aa$ (and over each other), hence at the cost of applying automorphisms, we can assume that they do not change during the construction.

Let $A_\mu=AI_{<\mu}J_{<\mu}$ and set $p_0(x) =\tp(a/A_\mu)$.

\medskip
\underline{Step 0}: Note that if some $(b,I,J,\psi;B)$ is linear, then so is $(b,I,J,\psi;B')$ for any $B'\supseteq B$ for which this quintuple is good. To prove the proposition it suffices then to find $A'\supseteq A$ such that the sequences $(I_\alpha+J_\alpha:\alpha<\mu)$ are mutually indiscernible over $A'$ and
$(a,I_0,J_0,\phi_0;A')$ is linear. Indeed, having done this, properties $\bullet_0$ and $\bullet_{2\to 4}$ still hold for $A'$ replacing $A$ (for $\bullet_4$, this is given by $\bullet'_4$) and we can enforce $\bullet_5$ as in the beginning of the proof of Lemma \ref{lem_max decomposition}. We can then inductively increase $A'$ to make each $(a,I_\alpha,J_\alpha,\phi_\alpha;A')$ linear one after the other.

\medskip
\underline{Step 1}: Set $\phi=\phi_0$ and $\mathbf u = (a,I_0,J_0,\phi;A)$.
Let $n$ be larger than the VC-dimension of the formula $\phi(x;y)$. We show that one cannot find tuples $a_k\models p_0$, $k<n$ and sequences $K_\alpha^ k$, $\alpha<\mu$, $k<n-1$ such that:

\begin{itemize}
\item[$\boxtimes_0$] the sequences $(I_\alpha+K_\alpha^0+\cdots+K_\alpha^{n-2}+J_\alpha:\alpha<\mu)$ are mutually indiscernible over $A$;

\item[$\boxtimes_1$] each $(a_k,I_\alpha+K_\alpha^{<k},K_\alpha^{\geq k}+J_\alpha,\phi_\alpha;A)$ is good;

\item[$\boxtimes_2$] for each $k\neq k'<n$, the tuples $a_k$ and $a_{k'}$ are $R(\mathbf u)$-incomparable.
\end{itemize}
Assume for a contradiction that we are given such tuples and sequences. We show that for any $\sigma$ permutation of $n$, we can find sequences $K_{0,\sigma}^k$, $k<n-1$ such that:

\begin{itemize}
\item[$\boxtimes_{3,\sigma}$] the sequence $I_0+K_{0,\sigma}^0+\cdots+K_{0,\sigma}^{n-2}+J_0$ is indiscernible over $A$;

\item[$\boxtimes_{4,\sigma}$] each $(a_k,I_0+K_{0,\sigma}^{<\sigma(k)},K_{0,\sigma}^{\geq \sigma(k)}+J_0,\phi)$ is good.
\end{itemize}
Why is this enough? Fix any $\sigma$ a permutation of $n$, $i<n$ and let $e$ be an element of $K_{0,\sigma}^{i}$ (or $J_0$ if $i=n-1$). Then we have $\models \phi(a_j,e) \iff \sigma(j)>i$. This shows that the set $\{a_0,\ldots,a_{n-1}\}$ is shattered by $\phi$ and contradicts the choice of $n$.

We now turn to the construction of $K_{0,\sigma}^k$. For $\sigma$ the identity, we can take $K_{0,\sigma}^k=K_0^k$. Assume that we have built those sequences for some value of $\sigma$. Let $i<n-1$ and set $\tau= (i, i+1)\circ \sigma$. We show how to build the sequences $K_{0,\tau}^k$. Since the sequences $(I_\alpha+J_\alpha:\alpha<\mu)$ are mutually indiscernible over $A$, the sequences $I_\alpha+K_\alpha^0+\cdots+K_\alpha^{n-2}+J_\alpha$, $0<\alpha<\mu$ and $I_0+K_{0,\sigma}^0+\cdots+K_{0,\sigma}^{n-2}+J_0$ are mutually stable over $A$. Replacing each of the sequences $I_\alpha,J_\alpha$, $K_\alpha^k$ and $K_{0,\sigma} ^k$ by a Morley sequence of their limit types and applying Lemma \ref{lem:poundpound}, we obtain:


\begin{itemize}
\item[$\boxtimes_{5,\sigma}$] the sequences $I_\alpha+K_\alpha^0+\cdots+K_\alpha^{n-2}+J_\alpha$, $0<\alpha<\mu$ and $I_0+K_{0,\sigma}^0+\cdots+K_{0,\sigma}^{n-2}+J_0$ are mutually indiscernible over $A$.
\end{itemize}
Set $u=\sigma^{-1}(i)$, $v=\sigma^{-1}(i+1)$ and let $b=a_{u}$ and $c=a_{v}$. Assume $u<v$ (the case $u>v$ is similar). The main thing to prove is:

\begin{itemize}
\item[$\boxplus$] The two sequences $I_0+K_{0,\sigma}^{<i}$ and $K_{0,\sigma}^{>i}+J_0$ are mutually indiscernible over $Abc$.
\end{itemize}

To prove this, let us first consider the situation over the base $Ab$. For $0<\alpha<\mu$, the tuple $(b,I_\alpha+K_\alpha^{<u},K_\alpha^{\geq u}+J_\alpha,\phi_\alpha;A)$ is good, and so is $(b,I_0+K_{0,\sigma}^{<i}, K_{0,\sigma}^{\geq i}+J_0,\phi;A)$. Hence by $\bullet_3$, the sequences $I_\alpha+K_\alpha^{<u}$, $K_\alpha^{\geq u}+J_\alpha$, $I_0+K_{0,\sigma}^{<i}$ and $K_{0,\sigma}^{\geq i}+J_0$, where $\alpha$ ranges in $0<\alpha<\mu$, are mutually stable over $Ab$. Similarly, the sequences $I_\alpha+K_\alpha^{<v}$, $K_\alpha^{\geq v}+J_\alpha$, $I_0+K_{0,\sigma}^{<i+1}$ and $K_{0,\sigma}^{\geq i+1}+J_0$ are mutually stable over $Ac$. Replacing all the sequences $K_\alpha^j$ and $K_{0,\sigma}^{j}$ by Morley sequences of their limit types over everything and applying Lemma \ref{lem:poundpound}, we can replace ``mutually stable" by ``mutually indiscernible" in the two previous sentences and obtain:

\begin{itemize}
\item[$\boxtimes_{6,\sigma}$] The sequences $I_\alpha+K_\alpha^{<u}$, $K_\alpha^{\geq u}+J_\alpha$, $I_0+K_{0,\sigma}^{<i}$ and $K_{0,\sigma}^{\geq i}+J_0$, where $\alpha$ ranges over $0<\alpha<\mu$, are mutually indiscernible over $Ab$.
\item[$\boxtimes_{7,\sigma}$] The sequences $I_\alpha+K_\alpha^{<v}$, $K_\alpha^{\geq v}+J_\alpha$, $I_0+K_{0,\sigma}^{<i+1}$ and $K_{0,\sigma}^{\geq i+1}+J_0$ where $\alpha$ ranges over $0<\alpha<\mu$, are mutually indiscernible over $Ac$
\end{itemize}

By $\boxtimes_{6,\sigma}$, the family \[(I_\alpha+K_\alpha^{<u}, K_\alpha^{u} : 0<\alpha<\mu){}^\frown (I_0+K_{0,\sigma}^{<i}, K_{0,\sigma}^i)\] has the same type as \[(I_\alpha, J_\alpha : 0<\alpha<\mu){}^\frown (I_0, J_0)\] over $Ab$. In particular, each of $(b,I_\alpha+K_\alpha^{<u},K_\alpha^{u},\phi_\alpha;A)$, $\alpha>0$, and $(b,I_0+K_{0,\sigma}^{<i},K_{0,\sigma}^{i},\phi;A)$ is good and together they have the same type as $(b,I_\alpha,J_\alpha,\phi_\alpha;A)$, $\alpha>0$, and $(b,I_0,J_0,\phi;A)$. In particular, they satisfy properties $\bullet_{2 \to 5}$.

The sequences $I_\alpha+K_\alpha^{<u}+K_\alpha^{u}$, $I_0+K_{0,\sigma}^{<i}+K_{0,\sigma}^{i}$ are mutually indiscernible over $AK^{>i}_{0,\sigma}J_0c$ (we are using $u<v$ here). Then applying $\bullet_3$ (where the $K$'s there are empty), we get that the sequences $(I_\alpha+K_\alpha^{<u},K_\alpha^{u} : 0<\alpha<\mu)$ along with $I_0+K_{0,\sigma}^{<i}$ and $K_{0,\sigma}^i$ are mutually indiscernible over $AK^{>i}_{0,\sigma}J_0 bc$. In particular:

\begin{itemize}
\item[] $I_0+K_{0,\sigma}^{<i}$ is indiscernible over $AK^{>i}_{0,\sigma}J_0bc$.
\end{itemize}

By a symmetric reasoning, interchanging the roles of $b$ and $c$, we get that $(K_\alpha^{v-1},K_\alpha^{\geq v}+J_0 : 0<\alpha<\mu)$ along with $K_{0,\sigma}^{i}$ and $K_{0,\sigma}^{
> i} +J_0$ are mutually indiscernible over $AI_0K^{<i}_{0,\sigma}bc$, and in particular:

\begin{itemize}
\item[] $K_{0,\sigma}^{> i} +J_0$ is indiscernible over $AI_0K^{<i}_{0,\sigma}bc$.
\end{itemize}

Thus $\boxplus$ follows from those two statements.


We can now finish the construction. By $\boxplus$, $(I_0+K_{0,\sigma}^{<i},K_{0,\sigma}^{>i}+J_0)$ and $(I_0,J_0)$ have the same type over $Abc$. As $b,c$ are not $R(\mathbf u)$-comparable, there is a sequence $K_{0,\tau}^{i}$ such that each of $(b,I_0+K_{0,\sigma}^{<i}+K_{0,\tau}^{i},K_{0,\sigma}^{>i}+J_0,\phi)$ and $(c,I_0+K_{0,\sigma}^{<i},K_{0,\tau}^{i}+K_{0,\sigma}^{>i}+J_0,\phi)$ are good. Set $K_{0,\tau}^{j}=K_{0,\sigma}^{j}$ for $j\neq i$. We claim that $\boxtimes_{3,\tau}$ and $\boxtimes_{4,\tau}$ are satisfied. The construction immediately gives $\boxtimes_{3,\tau}$ along with $\boxtimes_{4,\tau}$ where $k$ there is either $u$ or $v$. Take now $k$ different from $u$ and $v$. Assume that $\sigma(k)<i$ (the case $\sigma(k)>i+1$ is similar). Then $\boxtimes_{4,\tau}$ will follow for this value of $k$ if we know that $\neg\phi(a_k, b)$ holds for all $b\in K_{0,\tau}^i$. By the argument that lead to $\boxtimes_{6,\sigma}$, taking $a_k$ instead of $b$, we have that the sequence $K_{0,\sigma}^{i-1}+J_0$ is indiscernible over $Ab$. We also know that $\neg\phi(a_k,b)$ holds for $b$ in that sequence. It follows that $K_{0,\sigma}^{i-1}+K_{0,\tau} ^i+J_0$ is stable over $Ab$ and by Lemma \ref{lem:pound} it is indiscernible over it. Hence $\boxtimes_{4,\tau}$ follows.
\medskip
\\
\underline{Step 2}: Let $n$ be maximal such that there are $a_k$, $k<n$, and sequences $K_\alpha^ k$, $\alpha<\mu$, $k<n-1$, such that $\boxtimes_{0-2}$ above hold and let such elements and sequences be given. If $n=1$, set $b=a$, $A'=AI_{>0}J_{>0}$ and $I'_0=I_0$. If $n>1$, set $b=a_{n-1}$, $A'=AI_{>0}K_{>0}^{<n-1}J_{>0}I_0K_0^{<n-2}a_{<n-1}$ and $I'_0=K_0^{n-2}$. Let $\mathbf v=(b,I'_0,J_0,\phi;A')$. We show that $\mathbf v$ is linear. Assume not, then we can find some $c \equiv_{A'I'_0J_0} b$ such that $b$ and $c$ are $R(\mathbf v)$-incomparable. This means that there are sequences $L_0,L_1$ such that all of $(b,I'_0,L_0+J_0,\phi;A')$, $(c,I'_0+L_0,J_0,\phi;A')$, $(b,I'_0+L_1,J_0,\phi;A')$ and $(c,I'_0,L_1+J_0,\phi;A')$ are good. Then $c$ also satisfies $p_0$. Note that $I_0+L_0+J_0$ is indiscernible over $A$: if $n=1$ this is clear, if $n>1$, then $I'_0+L_0+J_0$ is indiscernible over $A'$, in particular over $AI_0$ and the results follows from the fact that $I_0+I'_0+J_0$ is indiscernible over $A$. The same is true for $I_0+L_1+J_0$. Therefore $c$ is $R(\mathbf u)$-incomparable with $b$, as witnessed by $L_0,L_1$. Since $c$ has the same type as $b$ over $A'I_0J_0$, it is also $R(\mathbf u)$-incomparable to all the $a_k$'s, $k<n-1$.

Construct, for $0<\alpha<\mu$, sequences $K_\alpha^{n-1}$ which are Morley sequences of $\lim(\op(J_\alpha))$ over everything built so far and each other.
%
%
%

We prove by induction on $k$ that for $k<n-1$:
\begin{itemize}
\item[$(P_k)$] $(K_\alpha^{\geq k}+J_\alpha:\alpha<\mu)$, are mutually indiscernible over $AI_{<\mu}K_{<\mu}^{<k}a_{\leq k}$.
\end{itemize}

First, by the argument of $\boxtimes_{6,\sigma}$, for each $k$ we have:
\begin{itemize}
\item[$\boxtimes'_6$] The sequences $(I_\alpha + K_\alpha^{<k}, K_\alpha ^{\geq k}+J_\alpha: \alpha<\mu)$ are mutually indiscernible over $Aa_k$.
\end{itemize}
Taking $k=0$, we obtain $(P_0)$. Assume $(P_{k-1})$ and we show $(P_k)$. The argument is the same as that used to show $\boxplus$ above. The tuples $(a_{k},K^{k-1}_\alpha, K^{\geq k}_\alpha + J_\alpha , \phi_\alpha; A)$, $\alpha<\mu$ have the same type all together as $(a,I_\alpha,J_\alpha,\phi_\alpha; A)$ and hence satisfy $\bullet_{2 \to 5}$. By induction hypothesis, the sequences $K_\alpha^{k-1}+K_\alpha^{\geq k}+J_\alpha$, $\alpha<\mu,$ are mutually indiscernible over $Aa_{\leq k-1}$. By $\bullet_3$, the sequences $(I_\alpha + K_\alpha^{<k}, K_\alpha ^{\geq k}+J_\alpha: \alpha<\mu)$ are mutually stable over $Aa_{\leq k-1}a_k$. By Lemma \ref{lem:poundpound}, those sequences are mutually indiscernible over $Aa_{\leq k}$ and $(P_{k})$ holds.

It follows that the sequences $K_\alpha^{n-2}+K_\alpha^{n-1}+J_\alpha$, $0<\alpha<\mu$ and $K_0^{n-1}+L_0+J_0$ are mutually stable over $AI_{<\mu}K_{<\mu} ^{<n-2}a_{< n}$ (if $n=1$, take $I_\alpha$ instead of $K_\alpha^{n-2}$ and remove $I_{<\mu}$ from the base.) Replacing those sequences by Morley sequences of their limit types if necessary and applying Lemma \ref{lem:poundpound}, we obtain:

\begin{itemize}
\item[$\boxplus'$] The sequences $K_\alpha^{n-2}+K_\alpha^{n-1}+J_\alpha$, $0<\alpha<\mu$ and $K_0^{n-1}+L_0+J_0$ are mutually indiscernible over $AI_{<\mu}K_{<\mu} ^{<n-2}a_{< n}$.
\end{itemize}

%
%
%
%
%
Decompose in an arbitrary non-trivial way $L_0$ as $L_0=L'+L''$. Then by $\boxplus'$, there is a point $d \models p(\mathbf v)$ such that $(d,I_\alpha+K_\alpha^{<n},J_\alpha,\phi;A)$ is good for all $0<\alpha<\mu$, as is $(d,I'_0+L',L''+J_0,\phi;A)$. Then neither $(d,b)$ nor $(c,d)$ are in $R(\mathbf u)$, as witnessed by $L'$ and $L''$ respectively. There is an infinite subsequence $L_1' \subseteq L_1$ such that either $(d,I'_0+L_1',J_0,\phi;A)$ or $(d,I'_0,L_1'+J_0,\phi;A)$ is good. In the first case, $(d,c)$ is not in $R(\mathbf u)$ and $d$ is $R(\mathbf u)$-incomparable with $c$. In this case, set $(a'_{n-1},a_n)=(c,d)$ and $K_0^{n-1}=L_1'$. In the second case, $(b,d)$ is not in $R(\mathbf u)$. Thus $d$ is $R(\mathbf u)$-incomparable with $b$. In that case, set $(a'_{n-1},a_n)=(b,d)$ and $K_0^{n-1}=L'$. Since $d$ satisfies $p(\mathbf v)$ it is also $R(\mathbf u)$-incomparable with each $a_k$, $k<n-1$. In both cases, the sequence $(a_0,\ldots,a_{n-2},a'_{n-1},a_n)$ along with $K_{<\mu}^{<n}$ contradicts the maximality of $n$.

\medskip
\noindent
\underline{Step 3}: We have shown that $\mathbf v=(b,I'_0,J_0,\phi;A')$ is linear. If $n=1$, then we are done. Otherwise, for $0<\alpha<\mu$, build inductively sequences $K_\alpha^{n-2}\unlhd K_\alpha^{n-2}+I'_\alpha$ and $J_\alpha \unlhd J'_\alpha+J_\alpha$, over everything constructed so far. Let also $J'_0=J_0$. Then the tuple $(b,(I'_\alpha,J'_\alpha)_{\alpha<\mu},A)$ has the same type as $(a,(I_\alpha,J_\alpha)_{\alpha<\mu},A)$. Let $\sigma$ be an automorphism sending the first tuple to the second. Then we can take $\sigma(A')$ as the $A'$ we need to finish the proof.
\end{proof}

In the following theorem, by a \emph{linear quasi-order}, we mean a reflexive and transitive relation for which any two points are comparable.

\begin{thm}\label{th:main}
Let $T$ be NIP, $p(x)$ any type with $\dlr(p)\geq \mu$. Then there is an extension $q\supseteq p$ over some set $A$, relations $R_\alpha(x,y)$, $\bigvee$-definable over $A$, such that each $R_\alpha$ defines a linear quasi-order $\leq_\alpha$ with infinite chains on the set of realizations of $q(x)$. Furthermore, those orders are dense and independent in the sense that if $a_\alpha <_\alpha b_\alpha$ are given for $\alpha<\mu$, then there is $c\models q$ such that $a_\alpha <_\alpha c <_\alpha b_\alpha$ for all $\alpha<\mu$.
\end{thm}
\begin{proof}
	Let $a\models p$. By the assumption that $\dlr(p)\geq \mu$, we can find some $\textbf u_\alpha = (a,I_\alpha,J_\alpha,\phi_\alpha;A)$ which are good and such that the sequences $I_\alpha+J_\alpha$ are mutually indiscernible over $A$. Using then Lemma \ref{lem_max decomposition} and Proposition \ref{prop_exists linear} and replacing $A$ by $A'$ there, we can assume that all the $\textbf u_\alpha$ are linear. Let then $q=\tp(a/AI_{<\mu} J_{<\mu})$ and $R_\alpha=R(\mathbf u_\alpha)$.
	
	To see that the independence condition holds, let $a_\alpha,b_\alpha\models q$ be given with $\neg R_\alpha(b_\alpha,a_\alpha)$ (that is $a_\alpha <_\alpha b_\alpha$). For each $\alpha<\mu$, let $K_\alpha$ be such that both $(a_\alpha, I_\alpha,K_\alpha +J_\alpha,\phi;A)$ and $(b_\alpha,I_\alpha+K_\alpha,J_\alpha,\phi;A)$ are good. The sequences $(I_\alpha+J_\alpha:\alpha<\mu)$ are mutually indiscernible over $A$, hence the sequences $(I_\alpha+K_\alpha+J_\alpha:\alpha<\mu)$ are mutually stable over $A$. Replacing the sequences $K_\alpha$ by Morley sequences of their limit types, we can assume that the sequences $(I_\alpha+K_\alpha+J_\alpha : \alpha<\mu)$ are mutually indiscernible over $A$. Decompose each $K_\alpha$ into two infinite pieces as $K_\alpha = K^0_\alpha+K^1_\alpha$. By mutual indiscernability, there is $c\models q$ such that all $(c,I_\alpha+K^0_\alpha,K^1_\alpha+J_\alpha,\phi;A)$ are good. We then have $a_\alpha<_\alpha c <_\alpha b_\alpha$. This proves both density and independence.
	
	Note that if we carry out this construction without $b_\alpha$, we get $c$ such that $a_\alpha <_\alpha c$ for each $\alpha$. This proves that for each order $<_\alpha$ there are two realizations of $q$ which are strictly comparable. Then by density, each order $<_\alpha$ has infinite chains on realizations of $q(x)$.
\end{proof}

In the case where $\mu$ in the above theorem is finite, then we can modify this result to have $A$ be finite, at the cost of weakening independence. This boils down to a simple compactness argument. However, we first need to change slightly the notion of linearity. The reason is that with the previous definitions, we considered sequences $K$ with a certain type over $AIJ$, hence the base contains $I$ and $J$ and is thus always infinite. In order to be able to have the base be finite, we need to remove $I$ and $J$ from it. To this end, we introduce the following definitions.

\ignore{
\begin{definition}
A quintuple $\mathbf u = (\pi(x),I,J,\phi;A)$ is \emph{minimal$_0$} if:

$\bullet$ it is good;

$\bullet$ whenever $a\models \pi(x)$ and $L$  has the same EM-types as $I$ over $A$, if $L_0=\{c\in K: \models \phi(a;c)\}$ and $L_1=\{c\in K:\models \neg\phi(a;c)\}$, then $L=L_0+L_1$.
\end{definition}

}

\begin{definition}
	Let $\mathbf u = (\pi(x),I,J,\phi;A)$ be good. We define the following binary relations on realizations of $\pi(x)$:
	
	$\bullet$ $(a,b)\in E_0(\mathbf u)$ if for any infinite sequence $L$ of same EM-type as $I$ over $A$, if $\phi(a;e)$ holds for all $e\in L$, then $\phi(b;e)$ holds for almost all $e\in L$ and symmetrically if $\neg \phi(a;e)$ holds for all $e\in L$, then $\neg\phi(b;e)$ holds for almost all $e\in L$;
			
	$\bullet$ $(a,b)\in R_0(\mathbf u)$ if we cannot find an infinite sequence $L$ of same EM-type as $I$ over $A$ such that both $\phi(a;e)$ and $\neg\phi(b;e)$ holds for all $e\in L$.
\end{definition}
	
We have as previously that $E_0(\mathbf u)$ is a $\bigvee$-definable equivalence relation on realizations of $\pi(x)$, that $R_0(\mathbf u)$ is a $\bigvee$-definable reflexive, transitive, $E_0(\mathbf u)$-equivariant, relation. Furthermore we have $E_0(a,b) \iff R_0(a,b)\wedge R_0(b,a)$.

We say that $\mathbf u$ is linear$_0$ if any two realizations of $p(\mathbf u)$ are $R_0(\mathbf u)$-comparable.

The following properties follow at once from the definitions.

\begin{lemma}\label{lem:normal to 0}
Let $\mathbf u=(a,I,J,\phi;A)$ be good and build $\mathbf u' = (a,I',J',\phi;AIJ)$. where $I \unlhd_{AIJa} I+I'$ and $J\unlhd_{AII'Ja} J'+J$.

Then:
\begin{enumerate}
\item $\mathbf u'$ is good;
\item $E(\mathbf u)=E_0(\mathbf u')$ and $R(\mathbf u)=R_0(\mathbf u')$ on realizations of $p(\mathbf u')$;
\item if $\mathbf u$ is linear, then $\mathbf u'$ is linear$_0$.
\end{enumerate}
\end{lemma}

The following lemma is true with linear instead of linear$_0$, but is truly useful only in the latter case.

\begin{lemma}\label{lem_reduce to finite base}
If $\mathbf u=(\pi(x),I,J,\phi;A)$ is linear$_0$ then there is a finite $A_0\subseteq A$ and a formula $\theta(x)\in \pi(x)$ such that $(\theta(x),I,J,\phi;A_0)$ is linear$_0$.
\end{lemma}
\begin{proof}

Saying that $(\pi(x),I,J,\phi;A)$ is linear$_0$ is saying that it is good and that one cannot find two infinite sequences $L=(c_i:i<\omega)$ and $K=(d_i:i<\omega)$, both having the same EM-type as $I$ over $A$ such that for $i<\omega$,
\[ \models \phi(a;c_i)\wedge \neg \phi(a';c_i)\wedge \neg \phi(a;d_i) \wedge \phi(a';d_i).\]

By compactness, one can find a finite $A_0\subseteq A$ and formula $\theta(x)\in \pi(x)$ such that this also holds with $A$ replaced by $A_0$ and $\pi(x)$ replaced by $\theta(x)$. Then $(\theta(x),I,J,\phi;A_0)$ is also good and is linear$_0$.
\end{proof}

Note that the goodness hypothesis is not needed for the compactness argument, but is necessary to ensure that the order we construct is not trivial.
\medskip

We now state our main theorem. We only state the result with finite bases for one order, to simplify the statement. The $n$-order version is below, stated under the $\omega$-categorical assumption.

\begin{thm}\label{th:mainfinite}
	Let $T$ be NIP, unstable. Then there is a finite set $A$, a formula $\theta(x)$ over $A$ and a relation $R(x,y)$ $\bigvee$-definable over $A$ which defines a dense linear quasi-order on $\theta(x)$ with an infinite chain.
\end{thm}
\begin{proof}
	As $T$ is unstable, the op-dimension of $x=x$ is at least 1 and we can find some good quintuple $\mathbf u=(a,I,J,\phi;A)$. By Proposition \ref{prop_exists linear}, up to increasing $A$, we can take $\mathbf u$ to be linear. By Lemma \ref{lem:normal to 0}, increasing $A$ some more, we can assume that $\mathbf u$ is linear$_0$. Then by Lemma \ref{lem_reduce to finite base}, we get some $\mathbf u_0=(\theta(x),I,J,\phi;A_0)$ which is linear$_0$, with $A_0$ finite. Then $R=R_0(\mathbf u_0)$ is as required.
\end{proof}

\begin{thm}\label{th:mainomega}
If the theory $T$ is $\omega$-categorical, NIP, $\dlr(x=x)\geq n>0$, then there is a finite set $A_*$, an $A_*$-definable set $X$ and $n$ $A_*$-definable linear quasi-orders $\leq_1,\ldots,\leq_n$ on $p$, such that the structure $(X;\leq_1,\ldots,\leq_n)$ contains an isomorphic copy of every finite structure $(X_0;\leq_1,\ldots,\leq_n)$ equipped with $n$ \emph{linear orders}.
\end{thm}
\begin{proof}
	This is similar to Theorem \ref{th:mainfinite} except that we start with $n$ good quintuples $\mathbf u_k = (a,I_k,J_k,\phi;A)$, where $(I_k+J_k:k<n)$ are mutually indiscernible over $A$. By Proposition \ref{prop_exists linear}, we can assume that they are linear. By the same argument as in the proof of Theorem \ref{th:main}, the order $R(\mathbf u_k)$ are independent.
	
	As in Lemma \ref{lem:normal to 0}, we can successively construct $I_k \unlhd I_k+I'_k$ and $J_k \unlhd J'_k+J_k$, each time over everything constructed so far. Then replacing $I_k$ by $I'_k$, $J_k$ by $J'_k$ and $A$ by $AI_{<n}J_{<n}$ we can assume that the quintuples are linear$_0$. Note that after having done this substitution, the sequence $(I_k+J_k:k<n)$ are still mutually indiscernible over $A$. Lemma \ref{lem_reduce to finite base} then gives us linear$_0$ quintuples $\mathbf u^0_k = (\theta_k(x),I_k,J_k,\phi;A_k)$, with $A_k$ finite. Let $A_*=\bigcup_{k<n} A_k$ and $\theta(x)=\bigwedge_{k<n} \theta_k(x)$. Then each $\mathbf u^1_k:=(\theta(x),I_k,J_k,\phi;A_*)$ is linear$_0$ and we define the order $\leq_k$ to be given by the relation $R_0(\mathbf u^1_k)$.
	
	Note that if $(b,a)\notin R(\mathbf u_k)$, then also $(b,a)\notin R_0(\mathbf u^1_k)$ witnessed by the same sequences, and thus $a<_k b$. The statement about universality therefore follows from the independence of the orders $R(\mathbf u_k)$.
\end{proof}
\subsection{Theories with no interpretable linear order}

Having found a linear order, the natural next step would be to understand the induced structure on it. When the order is interpretable, this becomes an instance of the classical problem of studying NIP ordered structures. The dp-minimal case in particular has received some attention (see e.g. \cite{GoodMon}, \cite{dpmin}), though most results assume an ordered-group structure. The $\omega$-categorical case is considered in \cite{rank_one}. However, we expect that more often than not, the order we constructed will be strictly $\bigvee$-definable. It seems likely that one could actually take advantage of it as the non-definability limits the possibilities for the induced structure. We give an example of that here and leave further studies for later. We show that if the theory does not interpret any infinite linear order, then in some precise sense, the induced structure on the $\bigvee$-definable quotient is weakly o-minimal.

\smallskip
We work in a general context not relying on the previous notations. Let $D$ be a definable set over some $A$ and $S(x,y)$ a $\bigwedge$-definable relation over $A$ such that $S(x,y)\to D(x)\wedge D(y)$ and $S(x,y)$ is a strict linear quasi-order with infinite chains on $D$ that is:

-- $S$ is transiitve and anti-reflexive;

-- $E(x,y):=\neg S(x,y)\wedge \neg S(y,x)$ is (a $\bigvee$-definable) equivalence relation such that $S$ is $E$-equivariant;

-- $S$ induces an infinite strict linear order on the quotient $D/E$.

\smallskip\noindent
Assume also that $S$ is type-definable by a countable conjunction of formulas. (Note that we can always ensure this in our construction since we end up with a finite $A$ and we can work in a reduct to a countable sublanguage.)

We can then write $S(x,y)=\bigwedge_{i<\omega} S_i(x,y)$ such that:

$\bullet$ $S_0(x,y)\to D(x)\wedge D(y)$;

$\bullet$ $\neg (\exists x,y) S_0(x,y)\wedge S_0(y,x)$;

$\bullet$ $(\forall x,y) S_{i+1}(x,y)\to S_i(x,y)$;

$\bullet$ $(\forall x,y) S_{i+1}(x,y)\wedge S_{i+1}(y,z)\to S_i(x,z)$.

\smallskip
\noindent
Note that for all $i<\omega$, we have:
\[ (\forall x,y)\left [ (S(x,y)\wedge S_i(y,z)) \to S(x,z)\right ].\]

Why? Assume that $\neg S(x,z)$ and $S(x,y)$ hold. Then as $S$ is linear, we must have $S(z,y)$, hence $S_0(z,y)$ holds. By the second bullet above, this implies $\neg S_0(y,z)$ and hence $\neg S_i(y,z)$.

In particular:
\[(\square) \quad (\forall x,y)\left [ (S(x,y)\wedge S_i(y,z)) \to S_i(x,z)\right ].\]

Now for $i<\omega$, define $x \leq_i y$ as:
\[a\leq_i b \iff (\forall c\in D) S_i(b,c)\to S_i(a,c).\]

Then $\leq_i$ is a transitive, reflexive relation and by $(\square)$, for $a,b\in D$ we have \[S(a,b)\to a\leq_i b \to \neg S(b,a).\]

We also define $a<_i b \iff (a\leq_i b) \wedge \neg (b\leq_i a)$. This is a transitive, irreflexive relation.

\begin{prop}
Assume that there is a definable subset $X\subseteq D$ whose projection on $D/E$ is not a finite union of convex sets. Then there is an infinite interpretable linear order.
\end{prop}
\begin{proof}
Let $B$ be such that $X$ is defined over $B$. The assumption implies that for all $n<\omega$, we can find $a_1,b_1,\ldots,a_n,b_n\in D$ such that:

$\bullet$ if $a'_i$ is $E$-equivalent to $a_i$, then $a'_i \notin X$;

$\bullet$ $b_i\in X$ for all $i\leq n$;

$\bullet$ we have $S(a_1,b_1)\wedge S(b_1,a_2)\wedge S(a_2,b_2) \wedge \cdots \wedge S(a_n,b_n)$.

Consider the definable set $F=\{x\in D : (\forall y\in X) x<_0 y \vee y<_0 x\}$, that is the set of points strictly $\leq_0$-comparable to all points in $X$. Note that if $x\in D$ is not $E$-equivalent to any point in $X$, then $x\in F$, since it will even be $S$-comparable to all points in $X$. For $a\in F$, let $X(a)= \{x\in X : a<_0 x\}$.

\smallskip\noindent
\underline{Claim}: The sets $X(a)$, $a\in F$ are linearly ordered by inclusion.
\\
\emph{Proof}: Assume that $a,b\in F$ are such that $X(a)\nsubseteq X(b)$ and let $x\in X(a)\setminus X(b)$. We then have $a<_0 x$ and $\neg(b <_0 x)$. Since $b$ is in $F$, we must have $x<_0 b$. If there is $y\in X(b)\setminus X(a)$, then we have $a<_0 x<_0 b <_0 y <_0 a$ and $a<_0 a$ by transitivity of $<_0$, which is absurd. We conclude that $X(b)\subseteq X(a)$.

\smallskip
On $F$ we can define the equivalence relation \[E_X(a,b) \iff X(a)=X(b).\] By the assumptions on $X$, the quotient $F/E_X$ is infinite. It is also linearly ordered by $a\leq b \iff X(a)\subseteq X(b)$, which finishes the proof.\end{proof}

\section{Stable dimension}

This section is independent of the rest of the paper. We define the natural counterpart to op-dimension. We only show basic properties and leave its in depth study for later.
Throughout this section, we assume that $T$ is NIP.

\begin{definition}\label{definition: stable rank}
	Let $A$ be a set of parameters and $\pi(x)$ a partial type over $A$. We say that $\str(\pi,A)<\kappa$ if we cannot find the following:
	\begin{itemize}
		\item a tuple $a\models \pi$;
		\item infinite sequences $(I_i:i<\kappa)$ and $(J_i:i<\kappa)$ such that $(I_i+J_i:i<\kappa)$ are mutually indiscernible over $Aa$;
		\item tuples $(b_i:i<\kappa)$ such that $(I_i+b_i+J_i:i<\kappa)$ are mutually indiscernible over $A$, but for each $i<\kappa$, $I_i+b_i+J_i$ is not indiscernible over $Aa$.
	\end{itemize}
\end{definition}

Note that if we have such a witness to $\str(\pi,A)<\kappa$, then we can build one where the sequences have any given order type, by replacing them by Morley sequences of one of their limit types. In particular, we can ask for them to be indexed by $\mathbb Q$.

The following follows at once from the definitions.

\begin{lemma}
	If $\dpr(\pi,A)<\kappa$, then $\str(\pi,A)<\kappa$. In particular, if $T$ is NIP, then $\str(\pi,A)<|T|^+$.
\end{lemma}

The base change lemma is slightly harder to prove than for dp-rank. We start with a basic lemma about NIP.

\begin{lemme}\label{lem:mut ind over A}
For $i<\alpha$, let $I_i = (a_{i,j}:j<\beta)$ be an indiscernible sequence. Assume that the sequences $(I_i:i<\alpha)$ and mutually indiscernible over $\emptyset$ and the sequence of columns $((a_{i,j}:i<\alpha) : j<\beta)$ is indiscernible over $A$. Then the sequences $(I_i:i<\alpha)$ are mutually indiscernible over $A$.
\end{lemme}
\begin{proof}
To simplify notations, let us assume that $\alpha = 2$. The general case is similar. Start by increasing the sequence of columns to one indexed by $\mathbb Q$ of same EM-type over $A$.
If the sequences $I_0,I_1$ are not mutually indiscernible over $A$, then there is a formula $\phi(x_1,\ldots,x_n;y_1,\ldots,y_m)\in L(A)$ and four increasing tuples of elements of $\mathbb Q$: $\bar i := (i_1,\ldots,i_n)$, $\bar i' := (i'_1,\ldots,i'_n)$, $\bar j := (j_1,\ldots,j_m)$ and $\bar j' := (j'_1,\ldots,j'_m)$ such that $\models \phi(a_{0,\bar i},a_{1,\bar j})\wedge \neg \phi(a_{0,\bar i'},a_{1,\bar j'}).$

Now construct two sequences $\bar i_0 < \bar i_1 < \cdots $ and $\bar j_0 < \bar  j_1 < \cdots$ of increasing tuples of elements of $\mathbb Q$ so that for $k$ even $(\bar i_k,\bar j_k)$ has same order-type as $(\bar i,\bar j)$ and for $k$ odd, $(\bar i_k,\bar j_k)$ has same order-type as $(\bar i',\bar j')$. Then the sequence of $n+m$-tuples $(a_{0,\bar i_k}\hat{} a_{1,\bar j_k}:k<\omega)$ is indiscernible and the formula $\phi$ alternates infinitely on it. This contradicts NIP.
\end{proof}

\begin{lemma}\label{lem:base change stable}
	If $A\subseteq B$ and $\pi(x)$ is a partial type over $A$, then $\str(\pi,A)=\str(\pi,B)$.
\end{lemma}
\begin{proof}
	If $(a,I_i,J_i,b_i:i<\kappa)$ is a witness to $\str(\pi,B)\geq \kappa$, then $(a,I'_i,J'_i,b'_i:i<\kappa)$ witnesses $\str(\pi,A)\geq \kappa$, where $I'_i$ is obtained from $I_i$ by concatenating a fixed enumeration of $B$ to the end of every element of it, and same for $J'_i$ and $b'_i$.
	
	Conversely, assume that $(a,I_i,J_i,b_i:i<\kappa)$ is a witness to $\str(\pi,A)\geq \kappa$. By Ramsey and compactness, we can find sequences $I'_i,J'_i$ of same EM-type as $I_i,J_i$ over $Aa$ such that $(I'_i+J'_i:i<\kappa)$ are mutually indiscernible over $Ba$. Replacing $I_i,J_i$ by $I'_i,J'_i$, we can assume that $(I_i+J_i:i<\kappa)$ are mutually indiscernible over $Ba$. Without loss, all sequences are indexed by $\mathbb Q$. Increase each sequence to $I_i+J_i+K_{i,0}+K_{i,1}+\cdots$, preserving mutual indiscernibility over $Ba$. Then for each $n<\omega$, the family of pairs $((I_i+J_i+K_{i,<n},K_{i,\geq n}):i<\kappa)$ has the same type as $((I_i,J_i):i<\kappa)$ over $Ba$. For each such $n$, let $b_{i,n}$ be such that \[I_i,b_i,J_i \equiv_{Aa} I_i+J_i+K_{i,<n},b_{i,n},K_{i,\geq n}.\]
	Thinking of the sequences $(I_i+J_i+K_{i,<n}:i<\kappa)$ as rows of an array, the sequence of columns is indiscernible. Applying \cite[Lemma 2.8]{distal} to this sequence of columns, there is an automorphism $\sigma$ fixing $A,I_i,J_i,K_{i,<\omega}$ such that for each $n$, the sequence of columns of the array $(I_i+J_i+K_{i,<n}+\sigma(b_{i,n})+K_{i,\geq n}:i<\kappa)$ is indiscernible over $B$. By Lemma \ref{lem:mut ind over A}, for each $n$, the sequences $(I_i+J_i+K_{i,<n}+\sigma(b_{i,n})+K_{i,\geq n}:i<\kappa)$ are mutually indiscernible over $B$. Given a finite set $X\subseteq \kappa$, a finite set $\Delta$ of formulas and a finite $B_0\subseteq B$, by shrinking of indiscernibles in NIP (e.g. \cite[Proposition 3.32]{NIPbook}), there is $n=n(\Delta)$ such that the sequences $(K_{i,<n}+K_{i,\geq n}:i\in X)$ are mutually $\Delta$-indiscernible over $B_0\sigma(a)$. By compactness, we can find $(a',I'_i,J'_i,b'_i:i<\kappa)$ so that 
	\[a',I'_i,b'_i,J'_i \equiv_{A} a,I_i,b_i,J_i\]
	and $(I'_i+b'_i+J'_i:i<\kappa)$ are mutually indiscernible over $B$. This witnesses $\str(\pi, B)\geq \kappa$.
\end{proof}

We can now define $\str(\pi)$ as being equal to $\str(\pi,A)$ for some/any $A$ over which $\pi$ is defined. As usual, we define $\str(a/A)$ as $\str(\tp(a/A))$.

\begin{prop}\label{prop:strank}
	Let $\pi(x)$ be any partial type and $\kappa$ a cardinal. Then the following are equivalent:
	\begin{enumerate}
		\item $\str(\pi)<\kappa$;
		\item given $A$ over which $\pi$ is defined, $a\models \pi$ and dense endless sequences $( I_i:i<\kappa)$ which are mutually indiscernible over $A$ and mutually stable over $Aa$, there is $i<\kappa$ such that $I_i$ is indiscernible over $Aa$;
		\item 	given $A$ over which $\pi$ is defined, $a\models \pi$ and dense endless sequences $( I_i:i<\lambda)$ which are mutually indiscernible over $A$ and mutually stable over $Aa$, we can find $X\subseteq \lambda$, $|X|<\kappa$ such that the sequences $(I_i : i\in\lambda\setminus X)$ are mutually indiscernible over $Aa$.
	\end{enumerate}
\end{prop}
\begin{proof}
	$(1)\Rightarrow (2)$: Let $A,a$ and $(I_i:i<\kappa)$ be as in (2). Construct sequences $J_i^0, J_i^1$ indexed by $\mathbb Z$ so that $I_i \unlhd_{Aa} J_i^0+I_i+J_i^1$. By the assumption of mutual stability, the sequences $(J_i^0+J_i^1:i<\kappa)$ are mutually indiscernible over $Aa$. Assume that the conclusion of (2) fails. Then for each $i<\kappa$, there is an integer $n_i$ and a subsequence $\bar a_i$ of $I_i$ of size $n_i$ such that $J_i^0+\bar a_i+J_i^1$ is not indiscernible over $Aa$. If $J_i^0 = (b_i : i\in \mathbb Z)$, define ${J'}_i^0 = (b_{kn_i}{}^\frown \cdots {}^\frown b_{kn_i+n_i-1} : k\in \mathbb Z)$. Define ${J'}_i^1$ similarly. Having done this for all $i<\kappa$, we see that the family $({J'}_i^0+(\bar a_i)+{J'}_i^1:i<\kappa)$ witnesses $\str(\pi)\geq \kappa$.

	$(2)\Rightarrow (3)$: The argument is similar as for the analogous result for dp-rank from \cite{KOU} (also presented in \cite[Proposition 4.17]{NIPbook}). The case where $\kappa$ is infinite is rather straightforward: If the conclusion of (3) fails, we can construct inductively a sequence $(\delta_t:t<\kappa)$ of elements of $\lambda$ and a sequence $(\Delta_t:t<\kappa)$ of finite subsets of $\lambda$ such that:
	\begin{itemize}
		\item the sequence $I_{\delta_t}$ is not indiscernible over $Aa \cup \bigcup \{I_i : i\in \Delta_t\}$;
		\item the sets $\Delta'_t:= \Delta_t \cup \{\delta_t\}$, $t<\kappa$, are pairwise disjoint.
	\end{itemize}
	For $t<\kappa$, let $J_t$ be the sequence of columns of the array whose rows are the sequences $(I_i :i\in \Delta'_t)$. Since by hypothesis those rows are not mutually indiscernible over $Aa$, Lemma \ref{lem:mut ind over A} implies that the sequence $J_t$ is not indiscernible over $Aa$. Furthermore, the sequences $(J_t:t<\kappa)$ are mutually stable over $A$: take the same sequences witnessing mutual stability of the $I_t$'s (in the sense of Proposition \ref{prop:mutually stable} (1)). This shows that (2) fails.
		
	
	For $\kappa=n+1$ finite, we prove the result by induction on $\lambda$. If $\lambda\leq n$, then we can take $X=\lambda$. Assume that $\lambda=n+k+1$ is finite. Construct inductively endless sequences $J_i^0, J_i^1$ so that $I_i \unlhd J_i^0+I_i+J_i^1$, each one over everything built so far. Let $B=A \cup \bigcup \{J_i^0,J_i^1:i<\lambda\}$. Then the sequences $(I_i:i<\lambda)$ are mutually indiscernible over $B$ and mutually stable over $Ba$. By (2), there is $i(*)<\lambda$ such that $I_{i(*)}$ is indiscernible over $Ba$.
	
	\smallskip
	\underline{Claim}: The sequences $(I_i:i\neq i(*))$ are mutually stable over $AI_{i(*)}a$.
	
	\emph{Proof}: It suffices to show that the sequences $(J_i ^0+J_i ^1 : i\neq i(*))$ are mutually indiscernible over $AI_{i(*)}a$. If not, then there is some formula $\phi(\bar u; \bar v; \bar a)$ witnessing it, where $\bar u$ is a tuple of elements from $(J_i ^0+J_i ^1 : i\neq i(*))$, $\bar v$ a tuple of elements from $I_{i(*)}$ and $\bar a$ a tuple of elements from $Aa$. Now by mutual stability, $J^0_{i(*)}+J^1_{i(*)}$ is indiscernible over $(J_i ^0+J_i ^1 : i\neq i(*))\cup Aa$. By indiscernibility, $\bar u$ can be taken to be any tuple of the right order type from $I_{i(*)}$. Therefore by Fact \ref{fact:finite cofinite}, if we take $\bar u'$ in $J^0_{i(*)}$ of the same order type, we also have $\phi(\bar u; \bar v; \bar a)$. But this contradicts the fact that $(J_i ^0+J_i ^1 : i\neq i(*))$ are mutually indiscernible over $AaJ^0_{i(*)}J^1_{i(*)}$.
	
	\smallskip
	
	
	By induction hypothesis, working over the base set $AI_{i(*)}$, there is $X_0\subseteq X\setminus \{i(*)\}$ of size at most $n$ such that the sequences $(I_i:i\notin X_0\cup \{i(*)\})$ are mutually indiscernible over $AI_{i(*)}a$. If $I_{i(*)}$ is indiscernible over $Aa \cup (I_i:i\notin X_0\cup \{i(*)\})$, then we can take $X=X_0$. Assume that this is not the case and we reach a contradiction as in the previous claim. As the sequences $(I_i:i\notin X_0\cup \{i(*)\})$ are mutually indiscernible over $AI_{i(*)}a$ and the $J^0_i$'s are built over all the $I_i$'s and $Aa$, the parameters from $(I_i:i\notin X_0\cup \{i(*)\})$ needed to witness that $I_{i(*)}$ is not indiscernible can be taken in $(J^0_i:i\notin X_0\cup \{i(*)\})$ instead. But then $I_{i(*)}$ is not indiscernible over $Ba$: contradiction.

	
	Finally, the case of infinite $\lambda$ can be deduced easily from the finite case as in \cite[Proposition 4.17]{NIPbook}.
	
	$(3)\Rightarrow (1)$ is clear.
\end{proof}

\begin{prop}\label{proposition: additivity of st-rank}
	Let $A$ be any set of parameters and $a,b$ two tuples. If $\str(a/A)<\kappa_1$ and $\str(b/Aa)<\kappa_2$, then $\str(a,b/A)< \kappa_1 + \kappa_2 -1$.
\end{prop}
\begin{proof}
	Let $(I_i :i<\lambda)$ be mutually indiscernible over $A$ and mutually stable over $Aab$. We can find $X_1\subseteq \lambda$, $|X_1|<\kappa_1$ such that the sequences $(I_i:i\in \lambda\setminus X_1)$ are mutually indiscernible over $Aa$. Next, we find $X_2\subseteq \lambda$, $|X_2|<\kappa_2$ such that the sequences $(I_i:i\in \lambda\setminus (X_1\cup X_2))$ are mutually indiscernible over $Aab$. This shows that $\str(a,b/A)<\kappa_1+\kappa_2-1$.
\end{proof}

\section*{Acknowledgment}

Thanks to Sergei Starchenko for comments on a previous version of this paper and the anonymous referee for an extremely helpful report.

\bibliography{tout.bib}

\begin{thebibliography}{KOU13}

\bibitem[GH15]{GuinHill}
Vincent Guingona and Cameron~Donnay Hill.
\newblock On a common generalization of shelah's 2-rank, dp-rank, and o-minimal
  dimension.
\newblock {\em Annals of Pure and Applied Logic}, 166(4):502--525, 2015.

\bibitem[Goo10]{GoodMon}
John Goodrick.
\newblock A monotonicity theorem for dp-minimal densely ordered groups.
\newblock {\em J. Symb. Log.}, 75(1):221--238, 2010.

\bibitem[KOU13]{KOU}
Itay Kaplan, Alf Onshuus, and Alex Usvyatsov.
\newblock Additivity of the dp-rank.
\newblock {\em Transactions of the American Mathematical Society},
  365:5783--5804, 2013.

\bibitem[Pil96]{PillayBook}
A.~Pillay.
\newblock {\em Geometric stability theory}.
\newblock Oxford logic guides. Clarendon Press, 1996.

\bibitem[She71]{Sh10}
Saharon Shelah.
\newblock Stability, the f.c.p., and superstability; model theoretic properties
  of formulas in first order theory.
\newblock {\em Ann. Math. Logic}, 3(3):271--362, 1971.

\bibitem[She90]{Sh:c}
Saharon Shelah.
\newblock {\em Classification theory and the number of nonisomorphic models},
  volume~92 of {\em Studies in Logic and the Foundations of Mathematics}.
\newblock North-Holland Publishing Co., Amsterdam, second edition, 1990.

\bibitem[She14]{Sh863}
Saharon Shelah.
\newblock Strongly dependent theories.
\newblock {\em Israel Journal of Mathematics}, 204:1--83, 2014.

\bibitem[Sim11]{dpmin}
Pierre Simon.
\newblock On dp-minimal ordered structures.
\newblock {\em Journal of Symbolic Logic}, 76(4):448--460, 2011.

\bibitem[Sim13]{distal}
Pierre Simon.
\newblock Distal and non-distal theories.
\newblock {\em Annals of Pure and Applied Logic}, 164(3):294--318, 2013.

\bibitem[Sim15]{NIPbook}
Pierre Simon.
\newblock {\em A Guide to {NIP} theories}.
\newblock Lecture Notes in Logic. Cambridge University Press, 2015.

\bibitem[Sim18]{rank_one}
Pierre Simon.
\newblock {NIP} $\omega$-categorical structures: the rank 1 case.
\newblock preprint, 2018.

\bibitem[Sim20]{decomposition}
Pierre Simon.
\newblock Type decomposition in {NIP} theories.
\newblock {\em J. Eur. Math Soc.}, 22(2):455--476, 2020.

\bibitem[Usv09]{Us}
Alexander Usvyatsov.
\newblock On generically stable types in dependent theories.
\newblock {\em Journal of Symbolic Logic}, 74(1):216--250, 2009.

\end{thebibliography}

\end{document}